\documentclass[b5paper,reqno]{amsart}
\usepackage{amsfonts,lingmacros,tree-dvips, amsmath, amssymb, graphicx, mathrsfs}
\usepackage[hidelinks]{hyperref}

\usepackage[all]{xy}

\newdir{ >}{!/6pt/@{ }*:(1,0.2)@_{>}*:(1,-0.2)@^{>}}

\theoremstyle{plain}

\newtheorem{theorem}{Theorem}[section]

\theoremstyle{definition}
\newtheorem{definition}[theorem]{Definition}

\newtheorem{counter example}[theorem]{Counter Example}

\numberwithin{equation}{section}

\DeclareMathAlphabet{\mathscr}{OT1}{pzc}{m}{it} 
\begin{document}
\Large{
		\title{GENERALIZATION OF TWO THEOREMS OF STEINHAUS IN CATEGORY BASES}
		
		\author[S. Basu]{Sanjib Basu}
		\address{\large{Department of Mathematics,Bethune College,181 Bidhan Sarani}}
		\email{\large{sanjibbasu08@gmail.com}}
		
		\author[A.C.Pramanik]{Abhit Chandra Pramanik}
	\address{\large{Department of Pure Mathematics, University of Calcutta, 35, Ballygunge Circular Road, Kolkata 700019, West Bengal, India}}
	\email{\large{abhit.pramanik@gmail.com}}

		\thanks{The second author thanks the CSIR, New Delhi – 110001, India, for financial support}
	\begin{abstract}
	Here we unify two results of Steinhaus and their corresponding category analogues by extending them in the settings of category bases. We further show that in any perfect translation base, every abundant Baire set contains a full subset for which our second theorem fails.
	\end{abstract}
\thanks{}
	\maketitle

\section{INTRODUCTION}
In $[10]$, Steinhaus proved that the distance set $D(A,B)$ of two sets $A$ and $B$ of positive Lebesgue measure contains an interval where by distance set, we mean the set of all mutual distances between points of $A$ and $B$. In the same paper, while dealing with the subsets of the real line $\mathbb{R}$, he proved even stronger theorems, these are : 
\begin{theorem}
	If $\{A_n\}_{n=1}^\infty$ is any infinite sequence of Lebesgue measurable sets with positive measures, then there exists an infinite sequence $\{a_n\}_{n=1}^\infty$ of distinct points such that $a_n \in A_n  (n=1,2,3,...)$ and their mutual distances are all rationals.
\end{theorem}
\begin{theorem}
	E being an infinite Lebesgue measurable set, there exists an enumerable set $P$ composed of points whose mutual distances are all rational numbers and a set $Z$ of measure zero such that $P\subseteq E \subseteq P^{\prime} \cup Z$ where $P^{\prime}$ represents the derived set of $P$.
\end{theorem}

A set $A$ has the Baire property $[7]$ if it can be expressed as the symmetric difference $G\Delta P$ of an open set $G$ and a set $P$ of first category. Equivalently, if $A = (G - P) \cup Q$ where G is open and $P, Q$ are sets of first category. Steinhaus theorem on distance set has an exact category analogue in the realm of Baire category which is due to 
Picard [$8$]. He showed that if $A$ and $B$ are second category subsets of $\mathbb{R}$ with Baire property then their distance set contains an interval. Picard's theorem has been generalized by K.P.S and M. Bhaskara Rao [$1$] for sets in topological groups, by Kominek [$3$] in topological vector spaces and by Sander [$9$] in topological spaces with respect to the class of globally solvable mappings. The above two theorems (Theorem $1.1$ and Theorem $1.2$) were also given more general forms. Alongside with the measure theoretic results, these were set forth by Miller, Xenikakis and Polychronis [$4$] for set with Baire property in the real line in the light of certain specified classes of mappings $f:\mathbb{R}\times\mathbb{R}\mapsto\mathbb{R}$ each of which has first order continuous partial derivatives $f_x$ and $f_y$ non-vanishing on an open set containing $A\times A$ where $A$ is the union of sets in the sequence $\{A_n\}_{n=1}^\infty$.\\

In this paper, we prove two results each of which unifies  Theorem $1.1$ and Theorem $1.2$ with their corresponding category analogues in point-meager Baire category base having a countable pseudo base. Similar type of unifications were earlier established by Morgan [$5$] for perfect, translation category bases in the real line but ours is based on a different approach.

\section{PRELIMINARIES AND RESULTS}
All the definitions stated below are taken from [$5$].\\

By a category base,
\begin{definition}
We mean a pair $(X,\mathcal{C}$) where $X$ is a non-empty set and $\mathcal{C}$ is a family of subsets of $X$ non-empty members of which are called regions such that the following conditions are satisfied:
\begin{enumerate}
\item Every point of $X$ belongs to some region; i.e., $X = \cup$ $\mathcal{C}$.
\item Let $A$ be a region and $\mathcal{D}$ be any non-empty family of disjont regions having cardinality less than the cardinality of $\mathcal{C}$. Then\\
i) There exists a region $D\in\mathcal{D}$ such that $A\cap D$ contains a region provided  $A \cap ( \cup \mathcal{D}$) contains a region.\\
ii) There exists a region $B\subseteq A$ which is disjoint from every region in $\mathcal{D}$ provided  $A\cap(\cup \mathcal{D})$ contains no region.
\end{enumerate}
\end{definition}

In a category base ($X,\mathcal{C}$),
\begin{definition}
 A set $A$ is called \textbf{singular} if every region contains a region disjoint from $A$; \textbf{meager} if $A$ can be expressed as countable union of singular sets; \textbf{abundant} if $A$ is not meager; \textbf{co-meager} if it is the complement of a meager set; \textbf{Baire} if every region contains a subregion in which either $A$ or its complement (in $X$) is meager. 
\end{definition}

A category base is called \textbf{point-meager} if every singleton set in it is singular and a \textbf{Baire base} if every region in it is abundant.\\

Moreover, by a separable category base we mean a category base which contains a countable subfamily such that every region is abundant in at least one region from the subfamily. A subfamily of region having the property that every region contains at least one region from the subfamily is called a pseudo base. Having a countable pseudo base is a stronger condition than separability.\\

Let $\Phi$ be a set of one-to-one mappings of $X$ onto $X$ which is closed under formation of composition of mappings and formation of inverses.
\begin{definition}
	A family $\mathcal{S}$ of subsets of $X$ is invariant under $\Phi$, or, simply $\Phi$-invariant if $\phi$($\mathcal{S}$)=$\mathcal{S}$ for all $\phi$ $\in$ $\Phi$ where $\phi$($\mathcal{S}$)=\{$\phi$($S$): $S$$\in$$\mathcal{S}$\}
\end{definition}
It can be easily checked that in a category base ($X$,$\mathcal{C}$) if $\mathcal{C}$ is $\Phi$-invariant then so are the families of singular, meager and Baire sets.
\begin{definition}
	Any pair ($X$,$\mathcal{J}$) (or, briefly $\mathcal{J}$) is called a complete family if there exists a sequence $\{\Psi_n\}_{n=1}^\infty$,  $\Psi_n$ : $\mathcal{J}$ $\mapsto$ $\mathcal{J}$ satisfying 
	\begin{enumerate}
		\item  $\Psi_n(A)\subseteq A$ for every $A\in\mathcal{J}$ and $n=1,2,3,...$.
		\item $\bigcap\limits_{n=1}^{\infty}A_n\neq\phi$ for every sequence $\{A_n\}_{n=1}^\infty$ ($A_n\in\mathcal{J}$) for which the corresponding sequence $\{\Psi_n(A_n)\}_{n=1}^\infty$ is descending.
	\end{enumerate}
\end{definition} 
We now state our results
\begin{theorem}
	Let ($X,\mathcal{C}$) be a point-meager, Baire base in which meet of no two regions  (in case their intersection is non-empty) is a singular set. Moreover, let $\mathcal{C}$ be a complete family and $f:X\times X \mapsto X$ be a mapping such that for every $x,y\in X$, $f_x$ is one-to-one, $f^y$ is one-to-one and onto and  $\mathcal{C}$ is invariant under \{$f^y: y\in X$\} where $f_x:X\mapsto X$ and $f^y:X\mapsto X$ are the x-section and y-section of $f$. Let $D$ be a non-empty subset of $X$ which has  non-empty intersection with every region and also having the property that whenever \{$A,B\}\subseteq \mathcal{C}-\{\phi$\} there exists $\eta \in D$ such that $A\cap f^\eta(B)\neq \phi$. Then given any sequence  $\{A_n\}_{n=1}^\infty$ of abundant Baire sets in ($X,\mathcal{C}$), there exists a sequence \{$a_n\}_{n=1}^{\infty} (a_n\in A_n)$ and a sequence \{$\eta_n\}_{n=1}^{\infty}(\eta_n\in D)$ of distinct points satisfying the identity
	\begin{equation*}
		f(a_1,\eta_1)=f(a_2,\eta_2)=....=f(a_n,\eta_n)=...
	\end{equation*}
\end{theorem}
\begin{proof}
	In the given category base ($X,\mathcal{C}$), we choose regions $C_0,C_1,C_2$, $...,C_n$,... such that $C_0$ is arbitrary but fixed and $C_n -A_n$ is meager for $n=1,2,3,...$. As every abundant Baire set is abundant everywhere in some region and each $A_n$ is an abundant Baire set, this choice is justified. Let $C_n-A_n=\bigcup\limits_{i=1}^{\infty}F^{(n)}_i$ where $F^{(n)}_i$ are singular sets $(n,i=1,2,...)$. We now choose region $D_{11}\subseteq C_1-F^{(1)}_1$ and $\eta_1\in D$ such that $C_0\cap f^{\eta_1}(D_{11})\neq\phi$. Since by hypothesis the meet of no two regions can be singular, so by Theorem $2$ (Chapter $1$, Section II, [$5$]) there exists a region $H_{11}$ such that $\Psi_1(H_{11})\subseteq H_{11}\subseteq C_0\cap f^{\eta_1}(D_{11})$. But then  ${f^{\eta_1}}^{-1}(\Psi_{1}(H_{11}))\subseteq D_{11}$ $({f^{\eta_1}}^{-1}$ is the inverse of the function $f^{\eta_1})$ and ${f^{\eta_1}}^{-1}(\Psi_{1}(H_{11}))$ being a region according to our choice of function $f$, we may choose region $D_{12}$ such that $D_{12}\subseteq {f^{\eta_1}}^{-1}(\Psi_{1}(H_{11}))-F^{(1)}_2$, region $D_{21}$ such that $D_{21}\subseteq C_2-{F^{(2)}_1}$ and $f^{\eta_1}(D_{12})\cap f^{\eta_2}(D_{21}) \neq \phi$ for some $\eta_2\in D$. Consequently, by the same region as stated above, there exists a region $H_{22}$ such that $\Psi_2(H_{22})\subseteq H_{22}\subseteq f^{\eta_1}(D_{12})\cap f^{\eta_2}(D_{21})$. Now choose region $D_{13}$ such that $D_{13}\subseteq {f^{\eta_1}}^{-1}(\Psi_{2}(H_{22}))-F_3^{(1)}$, region $D_{22}\subseteq {f^{\eta_2}}^{-1}(f^{\eta_1}(D_{13}))-F_2^{(2)}$ and region $D_{31}\subseteq C_3-F_1^{(3)}$ such that $f^{\eta_1}(D_{13})\cap f^{\eta_2}(D_{22})\cap f^{\eta_3}(D_{31})\neq \phi$ for some $\eta_3\in D$. Consequently, there exists region $H_{33}$ such that $\Psi_{3}(H_{33})\subseteq H_{33}\subseteq f^{\eta_1}(D_{13})\cap f^{\eta_2}(D_{22})\cap f^{\eta_3}(D_{31})$. Now suppose after reaching the inclusion $\Psi_{n-1}(H_{n-1 n-1})\subseteq H_{n-1 n-1}\subseteq f^{\eta_1}(D_{1 n-1})\cap f^{\eta_2}(D_{2 n-2})\cap.....\cap f^{\eta_{n-1}}(D_{n-1 1})$, we further choose region $D_{1n}\subseteq {f^{\eta_1}}^{-1}(\Psi_{n-1}(H_{n-1 n-1}))-F_n^{(1)}$, region $D_{2n-1}\subseteq {f^{\eta_2}}^{-1}(f^{\eta_1}(D_{1n}))-F_{n-1}^{(1)}$,......and finally region $D_{n1}\subseteq C_n-F_1^{(n)}$ such that $f^{\eta_1}(D_{1n})\cap f^{\eta_2}(D_{2n-1})\cap.....\cap f^{\eta_{n-1}}(D_{n-12})\cap f^{\eta_n}(D_{n1})\neq \phi$ for some $\eta_n\in D$. Since our category base is a point-meager Baire base, the hypothesis ensures that we can choose the sequence \{$\eta_n\}_{n=1}^{\infty}$ in such a manner that all its members are mutually distinct. It is easy to check that the sequence $\Psi_1{(H_{11})},\Psi_2{(H_{22})},.....,\Psi_n{(H_{nn})},....$ is decreasing and so $\bigcap\limits_{n=1}^{\infty}H_{nn}\neq\phi$. But $\bigcap\limits_{n=1}^{\infty}H_{nn}\subseteq f^{\eta_1}(A_1)\cap f^{\eta_2}(A_2)\cap....\cap f^{\eta_n}(A_n)\cap...$. Hence there exists  a sequence \{$a_n\}_{n=1}^{\infty}$ ($a_n\in A_n$) made up of distinct terms (because $f_x$ is one-to-one ) such that \begin{equation*}
	f(a_1,\eta_1)=f(a_2,\eta_2)=....=f(a_n,\eta_n)=...	
	\end{equation*}			
\end{proof}
This finishes the proof.\\

In the following theorem, the notion of cluster point of a sequence has been used in the following sense. In a category base ($X,\mathcal{C}$), a point `$a$' is a cluster point of $E=\{a_n\}_{n=1}^{\infty}$ if for every region $C$ containg `$a$' and for every natural number m, there exists a natural number $p>m$ such that $a_p\in C$.
\begin{theorem}
	Let ($X,\mathcal{C}$) be a point meager, Baire base in which the conditions stated in the above theorem  are satisfied. Moreover, let ($X,\mathcal{C}$) has a countable pseudobase. Then given any abundant Baire set $A$, there exists 
	\begin{enumerate}
		\item a countably infinite subset $P=\{a_1,a_2,....,a_n,...$\} of $A$ and an infinite sequence \{$\eta_n\}_{n=1}^{\infty}$ of distinct terms satisfying \begin{equation*}
			f(a_1,\eta_1)=f(a_2,\eta_2)=....=f(a_n,\eta_n)=...
		\end{equation*}	
			and 
			\item a meager set $H$ such that \begin{equation*}
				 P\subseteq A\subseteq P^{\prime}\cup H
			\end{equation*}
		where $P^{\prime}$ denotes the set of all cluster points of $P$.
	\end{enumerate}
\begin{proof}
	Let $\mathcal{B}$ be a countable pseudobase. Since any category base possesssing a countable pseudobase is separable, it is easy to check that $X-\cup\mathcal{B}$ is meager. We set $A^*=A\cap(\cup\mathcal{B})$ and $A_n=A^*\cap B_n$ (n=1,2,3,...) where $\mathcal{B}=\{B_1,B_2,...,B_n,..\}$. We may assume that no $A_n$ is meager for otherwise, we may remove it without affecting our conclusions. By Theorem $2.5$, there exists infinite sequences \{$a_n\}_{n=1}^{\infty}$ ($a_n\in A_n$) and \{$\eta_n\}_{n=1}^{\infty}$ ($\eta_n\in D$) of distinct terms satisfying \begin{equation*}
			f(a_1,\eta_1)=f(a_2,\eta_2)=....=f(a_n,\eta_n)=...
	\end{equation*}
This proves ($1$).\\

We set $P=\{a_1,a_2,..,a_n,..$\} and claim that every point of $A^*$ is a cluster point of $P$. For any $x\in A^*$, we choose arbitrarily a region $C$ containing $x$ and a natural number $m$. Let $k_1$ be the least natural number such that $B_{k_1}\subseteq C$. If $k_1>m$, we put $p=k_1$. Then $p>m$ and $a_p\in A^*\cap B_p\subseteq C$. Otherwise, if $k_1\leq m$, we carry out the following procedure: choose $B_{k_2}\subseteq \Psi(B_{k_1})-\{a_{k_1}$\} and select $a_{k_2}\in A^*\cap B_{k_2}$ where $\Psi(B_{k_1})\subseteq B_{k_1}\subseteq C$. Again choose $B_{k_3}\subseteq \Psi(B_{k_2})-\{a_{k_2}$\} and select $a_{k_3}\in A^*\cap B_{k_3}$. Continuing in this manner where the choices are justified on account of the facts that the category base is point-meager and complete, we obtain by Archimedean property a natural number $k_j$ such that $k_j>m$. We put $k_j=p$ so that $a_p\in A^*\cap B_p\subseteq C$. This serves our purpose. Finally, setting $H=A\cap (X-\cup \mathcal{B}$) we get 
	$P\subseteq A\subseteq P^{\prime}\cup H$ which proves ($2$)
	\end{proof}
Hence the Theorem.
\end{theorem}
Earlier Morgan (Theorem $3, 5$, section II, Chapter $6$, [$5$]) unified the two theorems of Steinhaus (stated in the introduction) and the category analogues in perfect translation bases. From the fact that $f$ in the Theorem $2.5$ and Theorem $2.6$ may be chosen in particular as the function ($x,y$)$\mapsto x+y$ from $\mathbb{R}\times\mathbb{R}$ to $\mathbb{R}$, one may observe that our theorems too resemble those of Morgan in the sense that they also do the same unification under a different approach.\\

Following Grzegorek and Labuda [$2$], in any category base, we may say that a set $F$ is a full subset of a set $E$ if $F\subseteq E$ and for any Baire set $B$, $F\cap B$ is nonmeager whenever $E\cap B$ is so. If $E$ is a Baire set, this is equivalent to stating that $E-F$ can't contain any abundant Baire set. It is interesting to note that in a perfect translation base, given any abundant Baire set $B$, there exists a full subset of $B$ for which Theorem$2.6$ fails to hold. To prove this, we proceed as follows:\\

First we costruct a Vitali set $V$ in $\mathbb{R}$ which is also Bernstein in nature [$6$]. Since $\mathbb{R}=\bigcup\limits_{r\in \mathbb{Q}}(V+r)$ ($\mathbb{Q}$ is the set of rational), so for some $r_0\in\mathbb{Q}$, ($V+r_0)\cap B$ is abundant. Moreover, ($V+r_0)\cap B$ is a full subset of $B$, for otherwise, it would be possible to fit an abundant Baire set in $B$ which is disjoint from ($V+r_0)\cap B$. But this is impossible because every abundant Baire set in any perfect base contains a perfect set (Theorem $8$, Section II, Chapter $5$, [$5$]) and $V$ is Bernstein. Moreover, Theorem $2.6$ fails for the set ($V+r_0)\cap B$ is a subset of a Vitali set and therefore no two distinct points in it can be at a rational distance from one another.\\

Now ($V+r_o)\cap B$ is also a non-Baire set, for otherwise, by Theorem $3$ and Theorem $5$ (Section V, Chapter $6$, [$5$]), ($V+r_0)\cap B$ would not be an ARD set [$6$]. But non-Baireness of ($V+r_0)\cap B$ can also be derived without using Morgan's theorems. In fact, we can prove more, that no abundant  subset of ($V+r_0)\cap B$ can be Baire. If possible, let $E$ be an abundant Baire set contained in ($V+r_0)\cap B$. Choose $r^\prime\in\mathbb{Q}$ and fix it. Since $\mathbb{Q}-\{r^\prime\}$ is dense, by Theorem $3$ (Section I, Chapter $6$, [$5$]), $\bigcup\limits_{r\in{\mathbb{Q}-\{r^\prime\}}}(E+r)$ is co-meager. Consequently, the set $E+r^\prime$ which is disjoint with $\bigcup\limits_{r\in{\mathbb{Q}-\{r^\prime\}}}(E+r)$ is meager $-$ a contradiction.

\bibliographystyle{plain}

	\end{document}